\documentclass[oneside,a4paper,11pt,notitlepage]{article}
\usepackage[left=0.8in, right=0.8in, top=1.5in, bottom=1.5in]{geometry}

\usepackage{authblk}
\usepackage[T1]{fontenc} 
\usepackage[utf8]{inputenc} 
\usepackage[english]{babel}
\usepackage{lipsum} 
\usepackage{lmodern}
\usepackage{amssymb}
\usepackage{amsmath}
\usepackage{amsthm}
\usepackage{bm}
\usepackage{mathtools}
\usepackage{braket}
\usepackage{esint}
\newcommand{\abs}[1]{{\left|#1\right|}}

\usepackage{bigints}
\usepackage{enumitem}
\usepackage{booktabs}
\usepackage{graphicx}
\usepackage{tikz}
\usetikzlibrary{patterns}
\usepackage{multicol}
\usepackage{caption}
\usepackage{varwidth}
\usepackage{lipsum}
\usepackage{appendix}
\usepackage{tabularx}
\numberwithin{equation}{section}
\usepackage[skins,theorems]{tcolorbox}
\tcbset{highlight math style={enhanced,
		colframe=black,colback=white,arc=0pt,boxrule=1pt}}
\theoremstyle{definition}

\theoremstyle{plain}
\newtheorem{teorema}{Theorem}[section]

\newtheorem{lemma}[teorema]{Lemma}
\newtheorem{prop}[teorema]{Proposition}

\theoremstyle{definition}
\newtheorem{esempio}{Example}[section]
\newtheorem{oss}[esempio]{Remark}

\renewcommand{\det}{\text{det}}

\DeclareMathOperator{\R}{\mathbb{R}}

\newcommand{\Om}{\Omega}

\newcommand{\p}{\partial}
\newcommand{\nuout}{\nu}
\newcommand{\grad}{\nabla}
\newcommand{\Hh}{\mathcal{H}}

\newcommand{\Binom}[2]{\binom{#1}{#2}}
\DeclareMathOperator{\Div}{div}

\makeatletter

\newcommand{\myfootnote}[2]{\begingroup
	\def\@makefnmark{}%
	\addtocounter{footnote}{-1}%
	\footnote{\textbf{#1} #2}
	\endgroup}
\makeatother

\usepackage{hyperref}
\hypersetup{linktoc=none, bookmarksnumbered, colorlinks=true, linkcolor=magenta, citecolor=cyan}

\newcommand{\Addresses}{{
\bigskip 
  \footnotesize 
\noindent \textit{E-mail address}, Alba Lia~Masiello: \texttt{albalia.masiello@unina.it} 

     \medskip
\noindent \textit{E-mail address}, Gloria~Paoli (corresponding author): \texttt{gloria.paoli@unina.it} 

     \medskip

     \textsc{Dipartimento di Matematica e Applicazioni “Renato Caccioppoli”, Università degli Studi di
Napoli Federico II, Via Cintia, Monte S. Angelo, 80126 Napoli, Italy;}

\medskip
 \noindent\textit{E-mail address}, Francesco~Salerno: \texttt{f.salerno@ssmeridionale.it}

   \medskip 
 
  \textsc{Mathematical and Physical Sciences for Advanced Materials and Technologies, Scuola Superiore Meridionale, Largo San Marcellino 10, 80138 Napoli, Italy. }

 \par\nopagebreak 

}}

\title{On some functionals for which the ball is a saddle shape}
\author{Alba Lia Masiello, Gloria Paoli, Francesco Salerno}
\date{}

\begin{document}
\maketitle
\begin{abstract}
    In the present paper, we study the maximization  problem of the $k$-Torsional rigidity under quermassintegral constraint, with particular emphasis on the role of the ball. Our main result shows a phenomenon which, as far as we know, has not previously been observed: the ball, rather than being an extremal shape, is a saddle point. More precisely, for suitable quermassintegral constraints, we show that the $k$-torsional rigidity admits both increasing and decreasing perturbations around the ball. The proof is based on a second-order analysis of the corresponding variational problem. To the best of our knowledge, this is the first use of a second-order shape derivative approach in the study of variational problems for $k$-Hessian operators. 
    
    We also establish an analogous saddle-point behavior for a related curvature functional under a volume constraint, highlighting a structural distinction between the intermediate Hessian regime and the classical Laplace and Monge–Ampère cases.
    \newline
    \newline
    \textsc{Keywords:} $k$-Hessian equation, mixed volumes, shape derivative.  \\
    \textsc{MSC 2020:}   52A39, 35J60, 35J96, 49K20.
\end{abstract}

\section{Introduction}
Let $\Omega\subset\mathbb{R}^n$ be an open, \emph{convex} set  of class $C^2$. For $k=1,\dots, n$, we consider the $k$-Torsional rigidity 

    \begin{equation}
        \label{KTorsDef}
        T_k(\Omega) = \max\left\{\frac{\left(\displaystyle{\int_\Omega -v \, dx }\right)^{k+1}}{\displaystyle{\int_\Omega -v S_k(D^2v) \, dx }}, \quad  \text{ $v\in C^2(\Omega)$, $v$ $k$-convex, $v|_{\partial\Omega}=0$,}\, v\not\equiv 0 \right\}, 
    \end{equation}
   where $S_k(D^2 v)$ is the $k$-Hessian operator, defined as
\begin{equation}\label{kess}
    S_k(D^2 v)= \sum_{1\le i_1< \dots < i_k\le n} \lambda_{i_1}\cdots\lambda_{i_k}, \quad k=1,\dots, n,
\end{equation}
 $\lambda_i$ are the eigenvalues of the Hessian matrix of $v$ and a function $v\in C^2(\Omega)$ is said $k$-convex, if $S_i(D^2v)\geq 0$ for $i=1,2,...,k$. 
The supremum in \eqref{KTorsDef} is achieved by the solution to
    \begin{equation}
    \label{torsDirich}
        \begin{cases}
            S_k(D^2u)= \binom{n}{k}& \text{ in $\Omega$, }\\
            u=0 & \text{ on $\partial \Omega$},
        \end{cases} 
    \end{equation}
and consequently 
  $$T_k(\Omega)=\binom{n}{k}^{-1} \left(\int_\Omega-u\, dx\right)^{k}.$$

The $k$-Hessian operator $S_k$ is a second-order differential operator and it reduces to the Monge-Ampère operator for $k=n$ and to the Laplace operator for $k=1$, when the functional $T_1$ coincides with classical \emph{Torsional Rigidity} 

\begin{equation*}
    T(\Omega) = \sup_{v\in H^1_0(\Omega)}\frac{\displaystyle\bigg(\int_\Omega v\,dx\bigg)^2}{\displaystyle \int_\Omega\abs{\nabla v}^2\,dx}.
\end{equation*}

The famous Saint-Venant inequality asserts that among open, bounded sets in $\R^n$,  $T(\Omega)$ is maximized by any ball in $\R^n$ if the measure is prescribed, that is
\begin{equation}\label{1w0}
    T(\Omega)\le T(B), \quad \quad \text{ if } \abs{\Omega}=\abs{B}.
\end{equation}
This famous inequality is named after the mathematician who conjectured its validity (see \cite{SaintVenant1855}), but it was proved by P\'olya in 1948 \cite{Polya},  by means of symmetrization techniques. 

When $k\neq 1$, the optimization problem is more delicate, and it involves other geometrical quantities that are well defined when the set $\Omega$ is convex, the so-called 
quermassintegrals. Roughly speaking, the quermassintegrals $W_i(\Omega)$, for $i=1,\dots, n$, are intrinsic geometrical quantities that describe a convex set $\Omega$, in particular $W_0(\Omega) = \abs{\Omega}$ is the Lebesgue measure,  $nW_1(\Omega) = P(\Omega)$ is the perimeter, and $W_n(\Omega) = \omega_n$  is the measure of the Euclidean unit ball $B$.

In the case $k=n$, when the Hessian operator reduces to the Monge-Ampère operator, Talenti \cite{talenti_monge_ampere} in the planar case, 
and Tso in \cite{tso} in any dimension, prove that among convex sets of prescribed $(n-1)$-quermassintegral, the ball maximizes $T_n$. The proof of this result is obtained, once again, via a  symmetrization procedure: starting from the solution to \eqref{kess} for $k=n$, it is possible to define a function whose level sets are balls with the same $(n-1)$-quermassintegral and it is such that the quotient in \eqref{KTorsDef} increases, obtaining

\begin{equation}
    \label{nwn1}
    T_n(\Omega)\le T_n(B), \quad \text{ if } \quad W_{n-1}(\Omega)=W_{n-1}(B).
\end{equation}

Later on, in \cite{BNT_newisoperimetric}, the optimization of $T_n$ when prescribing the volume was studied. To better understand their result, it is crucial to
observe that the Monge–Ampère operator is invariant under volume preserving affine transformations. This means that if
$A$ is a volume preserving affine map, then $T_n(A\Omega)=T_n(\Omega)$, hence a
ball and  any ellipsoid having the same measure have the same $n$-Torsional rigidity. The isoperimetric problem in this case reads as follows

\begin{equation}
    \label{nw0}
    T_n(\Omega)\ge T_n(\mathcal{E}), \quad \text{ if } \quad\abs{\Omega}=\abs{\mathcal{E}},
\end{equation}

and it is proved by a symmetrization argument. In the proofs of \eqref{nwn1} and \eqref{nw0}, it is crucial that the functions involved are convex, as it is natural to study problems of the Monge-Ampère type.

The case $1<k<n$ is more delicate. In \cite{tso}, Tso’s symmetrization method
applies to functions with convex level sets: if $v$ is such a function, one can construct a rearranged function $v^*_{k-1}$, whose level sets are balls with the same $(k-1)$-th quermassintegral, and such that the quotient

$$\frac{\left(\displaystyle{\int_\Omega -v \, dx }\right)^{k+1}}{\displaystyle{\int_\Omega -v S_k(D^2v) \, dx }}\le\frac{\left(\displaystyle{\int_{\Omega_{k-1}^*} -v_{k-1}^* \, dx }\right)^{k+1}}{\displaystyle{\int_{\Omega_{k-1}^*} -v_{k-1}^* S_k(D^2v_{k-1}^*) \, dx }} $$

increases. So, \emph{if} the solution to \eqref{kess} has convex level sets, then the following is true

\begin{equation}
    \label{kwk1}
     T_k(\Omega)\le T_k(B), \quad \text{ if } \quad W_{k-1}(\Omega)=W_{k-1}(B).
\end{equation}

At present, the convexity of level sets is known only in specific
cases, $k=1$ (see \cite{Borell, HNST}), $k=2$, that was proved in \cite{Ma2008TheCO, salani_logcon} for $n=3$ and in the most recent papers \cite{salani2026Bis, salani2026} for the general dimension, and, as already said, $k=n$ (see  \cite{H}).
Thus, in the intermediate Hessian regime, the isoperimetric problem \eqref{kwk1} is known to be valid when $k=2$, while in the other cases
its formulation depends on the interaction between the analytic structure of $S_k$ and the geometry of convex bodies. Very recently in \cite{salani2026Bis}, the convexity of the level sets of the solution to \eqref{torsDirich} was disproved for $k=3$ and $n=4$.

\vspace{2mm}
The purpose of the present paper is to investigate the general variational problem

\begin{equation}
    \label{goal1}
    \max \left\{  T_k(\Omega)\,:\, \Omega \text{ is convex and }  W_j(\Omega)=C \right\}, \quad \quad 1\le k\le n,\quad 0\le j\le n-1.
\end{equation}

For $k=1$, the $n$ problems are a consequence of the Saint-Venant inequality \eqref{1w0}, the monotonicity of the Torsional rigidity with respect to domain inclusion and the Aleksandrov-Fenchel inequalities
\begin{equation}
    \label{Aleksandrov_Fenchel_inequalities}
    \biggl( \frac{W_j(\Omega)}{\omega_n} \biggr)^{\frac{1}{n-j}} \geq \biggl( \frac{W_i(\Omega)}{\omega_n} \biggr)^{\frac{1}{n-i}}, \qquad 0 \leq i < j \leq n-1,
\end{equation}
where equality holds if and only if $\Omega$ is a ball.
Indeed, if $B^0$ and $B^j$ denote the balls such that $\abs{B^0}=\abs{\Omega}$, $W_j(B^j)=W_j(\Omega)$, then \eqref{Aleksandrov_Fenchel_inequalities}  implies $B^0\subseteq B^j$, and hence
\begin{equation*}\label{SaintFenchel}
    T_1(\Omega)\le T_1(B^0)\le T_1(B^j).    
\end{equation*}
In order to have some hint of what happens in the remaining cases $k\not=1$, we adopt a shape derivative approach.
We consider a parameter $\delta>0$ and a family of deformations of the unit ball $B$ in $\R^n$ 
$\Phi\,:\,[0,\delta)\times\R^n\rightarrow\R^n$, such that
    \begin{itemize}
        \item [(i)] $\Phi(t,\cdot)$, $\Phi^{-1}(t,\cdot)\in C^\infty(\mathbb{R}^n,\mathbb{R}^n)$ for all $t\in[0,\delta)$,
        \item [(ii)] the maps $t\mapsto \Phi(t,x)$, $t\mapsto \Phi^{-1}(t,x)$ belong to $C^\infty([0,\delta))$ for all $x\in\mathbb{R}^n$,
        \item [(iii)] $\Omega(0)=B$ and $\Omega(t)=\Phi(t,B)$ for all $t\in[0,\delta)$,
        \item [(iv)] $W_j(\Omega(t))=W_j(B)=\omega_n$.
    \end{itemize}
The first Theorem that we prove is

\begin{teorema}\label{derivata1}
Let $\Phi\,:\,[0,\delta)\times\R^n\rightarrow\R^n$ be a one-parameter family of transformations that, for $\delta>0$ small enough, satisfies $(i)-(iv)$. Then, it holds

\begin{equation*}
    \frac{d T_k(\Omega(t))}{dt}\vrule_{t=0}=0.
\end{equation*}
  
\end{teorema}
This first result ensures that the ball is a critical point for the $k$-Torsion when prescribing \emph{any} quermassintegral for $j=0,\dots, n-1$. To understand the nature of this critical point, we restrict our analysis on the class of ellipsoids, and we prove

\begin{teorema}\label{teo:ellissi}
    Among all the ellipsoids of given quermassintegral $W_j$, $j=0,\dots, n-1$, the ball maximizes the $k$-Torsional rigidity for $1\le k <n$. 
\end{teorema}

The proof of Theorem \ref{teo:ellissi} proceeds by establishing the corresponding isoperimetric property of the
ball among ellipsoids with fixed volume; the conclusion for arbitrary quermassintegral
constraints then follows from the Aleksandrov–Fenchel inequalities.

A natural route for extending this result from ellipsoids to general convex bodies leads to the curvature functional

\begin{equation}\label{kaffine}
    \min_{\abs{\Omega}=C} \int_{\partial\Omega} H_{k-1}^\frac{1}{k+1}\,d\Hh^{n-1}, \quad \quad 1\le k<n,
\end{equation}
where $H_{k-1}$ denotes the normalized $(k-1)$-th mean curvature, i.e.,

 \[
 H_0 = 1, \qquad \qquad H_j = \binom{n-1}{j}^{-1} \sum_{1 \leq i_1 < \ldots < i_j \leq n-1} \kappa_{i_1} \ldots \kappa_{i_j}, \qquad j = 1,\ldots,n-1,
 \]
where  $\kappa_1, \ldots, \kappa_{n-1}$ are the principal curvatures of $\partial \Omega$.

For $k=1$, this problem reduces to the classical isoperimetric inequality, while for $k=n$ the maximization problem is well posed and it is known as the \emph{Petty inequality} (also referred to as affine surface area inequality \cite{Petty}) that asserts
\begin{equation}\label{AffIsopIneq}
    \int_{\partial\Omega} H_{n-1}^\frac{1}{n+1}\,d\Hh^{n-1}\le n \omega_n^\frac{2}{n+1} \abs{\Omega}^\frac{n-1}{n+1},   
\end{equation}
with equality if and only if $\Om$ is an ellipsoid. In both cases, the isoperimetric and the Petty inequalities are key tools to prove the optimization problem of the corresponding Torsional rigidity, that is why problem \eqref{kaffine} comes into play. Indeed, in \cite{BNT2}, the authors give an alternative proof of \eqref{nw0} combining a shape derivative argument with the affine curvature flow and \eqref{AffIsopIneq}. The intermediate regime, however, exhibits a different variational behavior, as stated in the following result.
\begin{teorema}\label{teo_curvature_sella}
    Let $2\le k<n$. Then, the ball is a saddle shape for the  functional

    $$\int_{\partial\Omega} H_{k-1}^\frac{1}{k+1}\,d\Hh^{n-1}$$
    when the volume is prescribed. 
\end{teorema}

This theorem suggests that the same happens for the $k$-Torsional rigidity when prescribing any quermassintegral. Indeed, we prove the following

\begin{teorema}\label{teo_sella}
    Let $0\le j<k-1$ and $1\leq k<n$. Then, the ball is a saddle shape for the $k$-Torsional rigidity among the convex sets with prescribed $j$-th quermassintegral. Moreover, if $k=n$ and $0<j<n-1$, then the ball is a saddle shape for the $n$-Torsional rigidity among the convex sets with prescribed $j$-th quermassintegral.
\end{teorema}
To prove Theorems \ref{teo_curvature_sella}-\ref{teo_sella}, we compute the second derivative

$$ \frac{d^2}{dt^2}\int_{\partial\Omega(t)} H_{k-1}^\frac{1}{k+1}\vrule_{t=0}\,d\Hh^{n-1}, \quad \frac{d^2 T_k(\Omega(t))}{dt^2}\vrule_{t=0}$$
and we show that there exist deformations of the unit ball along which the second derivative is positive and other deformations along which is negative. To summarize, the optimization results for the $k$-Torsional rigidity with prescribed quermassintegral are collected in Table \ref{Table Tk}.

\vspace{2mm}

Even if the study of optimization problems for the eigenvalue and the Torsional rigidity of the $k$-Hessian operator can be found in several papers (see, for instance, \cite{BT2}, \cite{DPG_stab_eig}, \cite{DPG_skautovalori}, \cite{Nunziask}, \cite{ms1}, \cite{ms2}) none of this works is based on a shape derivative approach. The first order shape derivative of the $n$-Torsional rigidity was computed in \cite{BGNT_ellipsoids}
and it is linked to the study of overdetermined problems for the Monge-Ampère equation, also studied in \cite{BNST_alternative,enache,GMPP}, but, to the best of our knowledge, this is the first use of a second-order shape derivative
approach in the study of variational problems for $k$-Hessian operators.

\begin{table}[ht]
\centering
\renewcommand{\arraystretch}{1.8}
\begin{tabular}{|c|c|c|c|}
\hline
&
&
$\displaystyle \inf_{\substack{W_j(\Omega)=c}} T_k(\Omega)$ &
$\displaystyle \sup_{\substack{W_j(\Omega)=c}} T_k(\Omega)$ \\
\hline
$k=1$
&
$0\le j\le n-1$
&
$0$ [Remark \ref{min:prob}]
&
$B$ is the unique maximizer \cite{Polya, tso}
\\
\hline
$1<k<n$
&
$0\le j<k-1$
&
$0$ [Remark \ref{min:prob}]
&
\parbox[c]{6cm}{\centering
$<\infty$, but $B$ is a saddle shape\par
[Theorem \ref{teo_sella}]
}
\\
\cline{2-4}
&
$k-1\le j\le n-1$
&
$0$ [Remark \ref{min:prob}]
&
$B$ is the unique maximizer \cite{tso}
\\
\hline
$k=n$
&
$j=0$
&
$\mathcal{E}$ is the unique minimizer \cite{BNT_newisoperimetric}
&
$\Delta_n$ [conjectured \cite{NamQLe}]
\\
\cline{2-4}
&
$0<j<n-1$
&
$0$ [Remark \ref{min:prob}]
&
\parbox[c]{6cm}{\centering
$<\infty$, but $B$ is a saddle shape\par
[Theorem \ref{teo_sella}]
}
\\
\cline{2-4}
&
$j=n-1$
&
$0$ [Remark \ref{min:prob}]
&
$B$ is the unique maximizer \cite{tso}
\\
\hline
\end{tabular}
\vspace{0.2cm}
\caption{Known and conjectured extremal properties of $T_k$ under quermassintegral constraints.}
\label{Table Tk}
\end{table}

\vspace{3mm}

The paper is organized as follows. In Section \ref{sec:preliminaries}, we recall some basic properties of elementary symmetric functions, $k$-Hessian operator and quermassintegrals, while in Section \ref{sec:ellisse} we prove Theorem \ref{teo:ellissi}. In Section \ref{sec:shape} we write the Taylor expansion of the geometric functionals and we prove Theorem \ref{teo_curvature_sella}, while in Section \ref{sec:prove} we compute the first and second shape derivatives of the $k$-Torsional rigidity, and we prove Theorem \ref{teo_sella}.

\section{Preliminaries}\label{sec:preliminaries}

   

\subsection{Elementary symmetric functions and Newton tensors}

For $\lambda=(\lambda_1,\dots,\lambda_n)\in\R^n$, let
\[
 S_k(\lambda)=\sum_{1\le i_1<\cdots<i_k\le n}
 \lambda_{i_1}\cdots\lambda_{i_k},
 \qquad S_0(\lambda)=1.
\]
If $A$ is a symmetric matrix with eigenvalues $\lambda(A)$, we write $S_k(A)=S_k(\lambda(A))$. 

We recall that the \emph{generalized Kronecker delta} is defined as
\begin{equation*}
    \delta_{j_1\dots j_k}^{i_1\dots i_k}=\begin{cases}
        1 & \text{ if $i_1,\dots,i_k$ are distinct integers and are an even permutation of $j_1,\dots,j_k$},\\
        -1 & \text{ if $i_1,\dots,i_k$ are distinct integers and are an odd permutation of $j_1,\dots,j_k$},\\
        0 & \text{ in all other cases}.
    \end{cases}
\end{equation*}
The latter can be also seen in terms of an $n\times n$ determinant as follows
\begin{equation*}
    \delta_{j_1\dots j_k}^{i_1\dots i_k}=\begin{vmatrix}
            \delta_{j_1}^{i_1} & \delta_{j_2}^{i_1} & \dots & \delta_{j_k}^{i_1}\\
            \delta_{j_1}^{i_2} & \delta_{j_2}^{i_2} & \dots & \delta_{j_k}^{i_2}\\
            \vdots & \vdots & \ddots & \vdots\\
            \delta_{j_1}^{i_k} & \delta_{j_2}^{i_k} & \dots & \delta_{j_k}^{i_k}
\end{vmatrix}.
\end{equation*}
A contraction property that will be  useful is the following: let $0<p<k$ be two integers, then
\begin{equation}\label{contraction}
    \delta_{j_1\dots j_p}^{i_1\dots i_p}\delta_{j_1\dots j_k}^{i_1\dots i_k}=\frac{p!(n-k+p)!}{(n-k)!}\delta_{j_{p+1}\dots j_k}^{i_{p+1}\dots i_k}.
\end{equation}
In particular, if $p=1$, formula \eqref{contraction} reduces to
\begin{equation}\label{contr2}
    \delta_{j_l}^{i_l}\delta_{j_1\dots j_k}^{i_1\dots i_k}=(n-k+1)\delta_{j_1\dots j_{l-1}j_{l+1}\dots j_k}^{i_1\dots i_{l-1}i_{l+1}\dots i_k},
\end{equation}
where $l\in\{1,\dots,k\}.$

The Newton tensor is
\begin{equation*}
 [T_k(A)]_i^{\ j}
 =\frac1{k!}\delta_{ii_1\dots i_k}^{jj_1\dots j_k}
 A_{j_1}^{i_1}\cdots A_{j_k}^{i_k},
\end{equation*}
where  $A$ is a $n\times n$ matrix. Using the generalized Kronecker symbol, we can rewrite the $k$-Hessian operator as follows
\begin{equation}\label{eq:Sk-Kronecker}
 S_k(A)=\frac1{k!}\delta_{i_1\dots i_k}^{j_1\dots j_k}
 A_{j_1}^{i_1}\cdots A_{j_k}^{i_k}.
\end{equation}
Differentiating \eqref{eq:Sk-Kronecker}, we have
\begin{equation}\label{cofproof}
        S_k^{ij}(A)=\frac{\partial}{\partial A_j^i}\frac{1}{k!}\delta_{ii_1\dots i_{k-1}}^{jj_1\dots j_{k-1}}A_j^iA_{j_1}^{i_1}\dots A_{j_{k-1}}^{i_{k-1}}=\frac{1}{(k-1)!}\delta_{ii_1\dots i_{k-1}}^{jj_1\dots j_{k-1}}A_{j_1}^{i_1}\dots A_{j_{k-1}}^{i_{k-1}}=[T_{k-1}]_i^j(A).
\end{equation}
For a positive definite matrix $A$, Newton's inequalities read
\begin{equation}\label{eq:Newton-inequalities}
 \left(\frac{S_\ell(A)}{\binom n\ell}\right)^{1/\ell}
 \ge
 \left(\frac{S_k(A)}{\binom nk}\right)^{1/k},
 \qquad 1\le\ell<k\le n.
\end{equation}
The equality case, when $1\le\ell<k$, is $A=\lambda I$.
  \subsection{The \texorpdfstring{$k$}{k}-Hessian operator}
Let $\Omega$ be an open subset of $\R^n$ and let $u\in C^2(\Omega)$. The $k$-Hessian operator is the $k$-th elementary symmetric function of the Hessian matrix $D^2u$. For $k=1$, the $k$-Hessian operator reduces to the Laplace operator; for $k>1$, the $k$-Hessian operator is fully nonlinear and it is non-elliptic, unless one restricts to the class  of $k$-convex functions

$${\Gamma}_k(\Omega)= \left\{u\in C^2(\Omega) : S_i (D^2u)\ge 0 \, \text{ in } \, \Omega, \,i=1,\dots, k\right\}.$$
In \cite{Caffarelli_existence}, the existence of solutions to \eqref{kess} is investigated, establishing that the classical Dirichlet problem is solvable 
if and only if $\partial\Omega$ satisfies
$S_{k-1}(\kappa_1,\dots, \kappa_{n-1}) > c_0 > 0$
where $\kappa_j$ denotes the principal curvatures of $\partial\Omega$, oriented so that convex domains have
nonnegative curvatures.

We observe that the operator $S_k^{\frac{1}{k}}$ is homogeneous of degree 1, and if we denote by

$$S_k^{ij}(D^2u)=\frac{\partial}{\partial u_{ij}}S_k(D^2u),$$
the Euler identity for homogeneous functions gives
\begin{equation*}
    S_k(D^2u)=\frac{1}{k} S_k^{ij}(D^2u)u_{ij}.
\end{equation*}
A direct computation shows that $\left(S_k^{1j}(D^2u), \dots, S_k^{nj}(D^2u)\right)$ is divergence-free, hence $S_k(D^2u)$ can be written in divergence form

\begin{equation*}
    S_k(D^2u)=\frac{1}{k}\left(S_k^{ij}(D^2u) u_j\right)_i,
\end{equation*}
where the subscripts $i, j$ stand for partial differentiation. 

\subsection{Quermassintegrals and Mean Curvatures}

Let $\Omega\subset\mathbb{R}^n$ be an open, bounded set with $C^{2}$ boundary.
The $k$-th normalized mean curvature is defined as
\begin{equation*}
    H_{k}=\dfrac{S_{k}(\kappa_1, \dots, \kappa_{n-1})}{\binom{n-1}{k}},
\end{equation*}
where $\kappa_1, \cdots,\kappa_{n-1}$ are the principal curvatures of $\partial\Omega$ at the point $x \in \partial\Omega$.  We stress that $H_1$ and $H_{n-1}$ are the mean and the Gaussian curvature of $\partial \Omega$, respectively. We also adopt the following conventions:
\begin{equation*}
    H_0=S_0=1, \quad\quad H_n=0.
\end{equation*}
A set $\Omega$ is $k$-convex (strictly $k$-convex) for  $k\in \{1, \dots, n-1\}$ if $H_j$ is non-negative (positive) at every point $x\in \partial \Omega$, for $j=1,\dots, k$.
 
When $\Omega$ is a $C^2$ convex set, we define the $k$-th quermassintegral of $\Omega$  as
\begin{equation*}
    W_k(\Omega)=\dfrac{1}{n} \int_{\partial \Omega} H_{k-1} d\mathcal{H}^{n-1},
\end{equation*}
with the special cases  $W_0(\Omega)=\abs{\Omega}$ and $W_1(\Omega)=P(\Omega)/n$.\\
If $u\in C^2(\Omega)$ and $t$ is a regular point of $u$, on the boundary of $\{u\le t\}$ it is possible to link the $k$-Hessian operator with the $(k-1)$-th curvature

\begin{equation}
    \label{hksk}
   \binom{n-1}{k-1} H_{k-1}=\frac{S_k^{ij}(D^2u)u_iu_j}{\abs{\nabla u}^{k+1}}.
\end{equation}

\subsection{Some properties of the \texorpdfstring{$k$}{k}-Torsional rigidity}
\begin{prop}
    The function
    $$k\to \binom{n}{k}^\frac{1}{k}T_k(\Omega)^\frac{1}{k}$$
    is increasing.
\end{prop}
\begin{proof}
    Let $u_k$ be the solution to
    \begin{equation*}
        \begin{cases}
            S_k(D^2u_k)= \binom{n}{k}& \text{ in $\Omega$, }\\
            u_k=0 & \text{ on $\partial \Omega$},
        \end{cases} 
    \end{equation*}
    and let $u_\ell$ be the solution to
    \begin{equation*}
        \begin{cases}
            S_\ell(D^2u_\ell)= \binom{n}{\ell}& \text{ in $\Omega$, }\\
            u_\ell=0 & \text{ on $\partial \Omega$}.
        \end{cases} 
    \end{equation*}
   Let us assume that $k>\ell$, hence the Newton inequalities imply

   $$\frac{S_\ell(D^2u_k)}{\binom{n}{\ell}}\ge \left(\frac{S_k(D^2u_k)}{\binom{n}{k}}\right)^\frac{\ell}{k}=1,$$
   hence, $u_k$ is a supersolution to

 \begin{equation*}
        \begin{cases}
            S_\ell(D^2u_k)\ge \binom{n}{\ell}& \text{ in $\Omega$, }\\
            u_k=0 & \text{ on $\partial \Omega$}.
        \end{cases} 
    \end{equation*}
   As a consequence of the comparison principle for the $k$-Hessian equation (see \cite[Theorem 17.1]{gilbarg}), we have $u_\ell\ge u_k$, hence
    \begin{equation}\label{comparisonTkTl}
        \binom{n}{k}^\frac{1}{k}T_k(\Omega)^\frac{1}{k}=\int_\Omega -u_k\, dx\ge \int_\Omega -u_\ell\, dx=\binom{n}{\ell}^\frac{1}{\ell}T_\ell(\Omega)^\frac{1}{\ell}.
    \end{equation}
\end{proof}
\begin{oss}
    Combining \eqref{comparisonTkTl} with the Alexandrov-Fenchel inequalities \eqref{Aleksandrov_Fenchel_inequalities} we get that the functional $T_k(\Om)W_j(\Om)^{-\frac{k(n+2)}{n-j}}$ is bounded from above in the class of $C^2$ convex sets. Indeed, we have
    $$T_k(\Om)W_j(\Om)^{-\frac{k(n+2)}{n-j}}\le \omega_n^{-\frac{jk(n+2)}{n(n-j)}}\binom{n}{k}^{-1}T_n(\Omega)^\frac{k}{n}\abs{\Omega}^{-\frac{k(n+2)}{n}}.$$
    The right-hand side is an affine invariant functional and it always admits a maximum (see \cite[\S 10.3]{Schneider_2013} for the proof). 
    This property makes the study of \eqref{goal1} nontrivial in the class of convex sets.  
\end{oss}
Moreover, it holds
\begin{prop}\label{prop_mon}
    The $k$-Torsional rigidity is monotone increasing with respect to set inclusion.
\end{prop}
\begin{proof}
    Let $\Omega_1\subset \Omega_2$, and let $u_1$ and $u_2$ be the solutions to
 \begin{equation*}
        \begin{cases}
            S_k(D^2u_1)= \binom{n}{k}& \text{ in $\Omega_1$ }\\
            u_1=0 & \text{ on $\partial \Omega_1$}
        \end{cases} , \quad \begin{cases}
            S_k(D^2u_2)= \binom{n}{k}& \text{ in $\Omega_2$, }\\
            u_2=0 & \text{ on $\partial \Omega_2$}.
        \end{cases}  
    \end{equation*}

    In $\Omega_1$, both $u_1$ and $u_2$ solve the same equation, while on $\partial \Omega_1$ we have $u_2\le u_1$, hence as a consequence of the comparison principle for the $k$-Hessian equation (\cite[Theorem 17.1]{gilbarg}), we have

$$ \binom{n}{k}^\frac{1}{k}T_k(\Omega_2)^\frac{1}{k}=\int_{\Omega_2} -u_2\, dx\ge \int_{\Omega_1} -u_2\, dx \ge \int_{\Omega_1} -u_1\, dx= \binom{n}{k}^\frac{1}{k}T_k(\Omega_1)^\frac{1}{k}.$$
\end{proof}

\section{The optimization problem among ellipsoids}\label{sec:ellisse}

Let $A\in \mathcal{S}_n$, we can always assume $A=\text{diag}(A)$. The set
$$\mathcal{E}=\{\abs{Ax}<1\}$$
is a generic ellipsoid. If $A=\lambda \mathbb{I},$ the ellipsoid reduces to the ball of radius $R=1/\lambda$. 

We are now in position to prove Theorem \ref{teo:ellissi}. 

\begin{proof}[Proof of Theorem \ref{teo:ellissi}]
We prove that, if we prescribe the volume of the set $\mathcal{E}$, the ball maximizes $T_k$. The Aleksandrov-Fenchel inequalities \eqref{Aleksandrov_Fenchel_inequalities} and the monotonicity property of the $k$-Torsional rigidity Proposition \ref{prop_mon} will imply the claim. 

    Let us start by proving that the function 
\begin{equation*}
    u_k(x) =\frac{\abs{Ax}^2-1}{2}\left(\frac{\binom{n}{k}}{S_k(A^2)}\right)^\frac{1}{k}
\end{equation*}
is the solution to
\begin{equation*}
    \begin{cases}
        S_k(D^2u_k)=\binom{n}{k} & \text{ in } \mathcal{E}, \\
        u_k=0 & \text{ on } \partial\mathcal{E}. 
    \end{cases}
\end{equation*}

Clearly, $u_k=0$ on $\partial\mathcal{E}=\{\abs{Ax}=1\}$, moreover, by recalling the homogeneity of the operator $S_k$, we have
\begin{align*}
    D^2u_k= \left(\frac{\binom{n}{k}}{S_k(A^2)}\right)^\frac{1}{k} A^2 \quad \Rightarrow \quad
    S_k(D^2u_k)=\binom{n}{k}.   
\end{align*}
Now, let us compute the $k$-torsional rigidity of $\mathcal{E}$. 
\begin{equation} 
\label{ellips:torsion}
    \begin{aligned}
      \binom{n}{k}^{1/k}T_k(\mathcal{E})^\frac{1}{k}&= \int_{\mathcal{E}}-u_k=\left(\frac{\binom{n}{k}}{S_k(A^2)}\right)^\frac{1}{k}\int_{\mathcal{E}}\frac{1-\abs{Ax}^2}{2}
        = \left(\frac{\binom{n}{k}}{S_k(A^2)}\right)^\frac{1}{k}\int_{B_1}\frac{1-\abs{x}^2}{2} \det(A^{-1})\\
        &= \left(\frac{\binom{n}{k}}{S_k(A^2)}\right)^\frac{1}{k}\det(A^{-1}) \frac{\omega_n}{n+2}
        =\left(\frac{\binom{n}{k}}{S_k(A^2)}\right)^\frac{1}{k}\frac{\abs{\mathcal{E}}}{n+2}.
    \end{aligned}
\end{equation}

 Let us observe that $$\det(A)=(\det(A^{-1}))^{-1}=\left(\frac{\abs{\mathcal{E}}}{\omega_n}\right)^{-1},$$ and by Newton's inequalities \eqref{eq:Newton-inequalities} we have

 \begin{equation*}
     T_k(\mathcal{E})^\frac{1}{k}\le {\binom{n}{k}^{-1/k}} \left(\frac{1}{\det(A^2)}\right)^\frac{1}{n} \frac{\abs{\mathcal{E}}}{n+2}= {\binom{n}{k}^{-1/k}}\frac{1}{(n+2)\omega_n^\frac{2}{n}} \abs{\mathcal{E}}^\frac{n+2}{n},
 \end{equation*}
 with equality if and only if $A=\lambda \mathbb{I}$, so if and only if $\mathcal{E}$ is a ball.
\end{proof}

\begin{oss}\label{min:prob}
    As a consequence of the explicit expression of the $k$-Torsional rigidity of an ellipsoid \eqref{ellips:torsion}, it is clear that the \emph{minimization} problem of the $k$-Torsional rigidity, for $k=1,\dots,n-1$ is trivial. Indeed, if we prescribe the measure of the ellipsoids, it is always possible to choose a sequence of diagonal matrices $A_n$ (and hence, a sequence of ellipsoids) such that 
    $$\frac{\det(A_n^{-1})}{\left(S_k(A_n^2)\right)^\frac{1}{k}}\to 0.$$ 
    
    Moreover, for $k=1,\dots,n$, when we prescribe the quermassinegral $W_j$, $j\ge 1$,
    by the explicit expression of $T_k(\mathcal{E})$ and the Newton inequalities, one  can write

    \begin{equation*}
        T_k(\mathcal{E})W_j(\mathcal{E})^{-\frac{k(n+2)}{n-j}}\leq \frac{W_j(\mathcal{E})^{-\frac{k(n+2)}{n-j}}}{(n+2)^k\omega_n^\frac{2k}{n}}\binom{n}{k}^{-1}\abs{\mathcal{E}}^\frac{k(n+2)}{n},
    \end{equation*}
\end{oss}
and it is always possible to find a sequence of ellipsoids for which $W_j$ is fixed and  the measure goes to $0$.

\section{The curvature functionals}\label{sec:shape}
The present section deals with the computation of the Taylor expansion of the geometric quantities involved in the paper, that are the quermassintegrals and 

$$\int_{\partial\Omega(t)}
H_{j-1}^{\frac1{j+1}}\,d\mathcal H^{n-1}.$$

We consider a family of deformations of the unit ball $B$ in $\R^n$ $\Phi:[0,\delta)\times\mathbb{R}^n\rightarrow\mathbb{R}^n$,  satisfying $(i)$ -- $(iv)$. We set
\begin{equation}
    \label{omegat}
    \Om(t)=\Phi(t,B_1)\subset\R^n,
\end{equation}
 where $B_1=\Om(0)$. 
Without loss of generality we can assume that the velocity field $\displaystyle{\frac{\partial\Phi}{\partial t}\vrule_{t=0}} $ is orthogonal to $\partial\Omega(t)$ and for all $x \in \partial\Omega(t)$  we denote by 
\[
        V=\nu \cdot\frac{\partial\Phi}{\partial t}\vrule_{t=0}, \qquad A=\nu \cdot\frac{\partial^2\Phi}{\partial t^2}\vrule_{t=0} .
\]
respectively the initial scalar velocity and the projection of the initial acceleration along the unit outer normal $\nu$ of $\partial\Omega(t)$. Under these assumptions, for $t$ small enough, the boundary of $\partial\Omega(t)$ can be represented in polar coordinates as

\begin{equation}
    \label{polarcoordinates}
    r(\xi, t) = 1 + V(\xi)t + A(\xi) \frac{t^2}{2} + o(t^2).
\end{equation}

\subsection*{Quermassintegrals expansion}
We recall the expansion of the $j$-th mean curvature that was proved in \cite{VW}.
\begin{lemma}[\cite{VW}]\label{exp-curv}
Let $\Omega\subset\mathbb{R}^n$. If there exists a function $r\in C^2(\partial B_1)$ such that $\partial\Omega=\left\{r(\xi)\xi: \xi\in \partial B_1\right\}$, then we have
    \begin{equation*}
    \begin{aligned}
        H_j&=\binom{n-1}{j}^{-1}(r^2+|\nabla_\xi r|^2)^{-\frac{j+2}{2}}\\
        &\times\sum_{m=0}^j(-1)^m\binom{n-1-m}{j-m}r^{-m}\left(r^2S_m(D^2r)+\frac{n-1+j-2m}{n-1-m}r_ir_l[T_m]_i^l(D^2r) \right),
        \end{aligned}
    \end{equation*}
    where 
    $r_i=\frac{\partial r}{\partial\xi_i}$.
\end{lemma}

Thanks to this result, we can prove the following.

\begin{teorema}\label{exp-querm}
    Let $\{\Omega(t)\}_t$ be the family of open, bounded convex sets of $\R^n$ defined in \eqref{omegat}, with $\Omega(0)=B_1$. Then, we can write the quermassintegrals as follows
    \begin{equation*}
        \begin{split}
            W_{j}(\Omega(t))&=\omega_n+\frac{t(n-j)}{n}\int_{\partial B_1}V(\xi)\,d\mathcal{H}^{n-1}_\xi\\
            &+\frac{t^2}{2n}\int_{\partial B_1}\left[(n^2+j^2+j-n-2nj)V(\xi)^2+\frac{nj-j^2}{n-1}\abs{\nabla_\xi V}^2+(n-j)A(\xi)\right]\,d\mathcal{H}^{n-1}_\xi+o(t^2),
        \end{split}
    \end{equation*}
    for $j=1,\dots,n-1$, and in particular
    \begin{equation*}
        W_0(\Omega(t))=\omega_n+t\int_{\partial B_1}V(\xi)\,d\mathcal{H}^{n-1}_\xi+\frac{t^2}{2}\int_{\partial B_1}\left[A(\xi)+(n-1)V(\xi)^2 \right]\,d\mathcal{H}^{n-1}_\xi+o(t^2).
    \end{equation*}
\end{teorema}
\begin{proof}
As we are representing the boundary of $\Omega(t)$ in polar coordinates, Lemma \ref{exp-curv} is in force, and we can write
\begin{equation}\label{parwj}
    \begin{split}
        &W_{j}(\Omega(t))=\frac{1}{n}\int_{\partial\Omega(t)}H_{j-1}\,d\mathcal{H}^{n-1}\\
        &=\frac{1}{n\binom{n-1}{j-1}}\int_{\partial B_1}\sum_{m=0}^{j-1}\frac{(-1)^m\binom{n-1-m}{j-1-m}r^{n-m-2}}{(r^2+|\nabla_\xi r|^2)^\frac{j}{2}}\left(r^2S_m(D^2r)+\frac{n+j-2-2m}{n-1-m}r_ir_l[T_m]_i^l(D^2r) \right)\,d\mathcal{H}_\xi^{n-1}\hspace{-3mm},
    \end{split}
\end{equation}
where we expressed the surface element as 
\[
d\mathcal H^{n-1}_{\partial\Omega(t)}
=
r^{n-2}
\bigl(r^2+|\nabla_\xi r|^2\bigr)^{1/2}
d\mathcal H^{n-1}_\xi.
\]

We observe that, when we represent the boundary $\partial\Omega(t)$ as in \eqref{polarcoordinates},
the second term in the sum in \eqref{parwj} is of higher order in $t$, except when $m=0$, that gives us
\begin{equation*}
    \binom{n-1}{j-1}\frac{n+j-2}{n-1}r^{n-2}\delta_i^j r_i r_j=\binom{n-1}{j-1}\frac{n+j-2}{n-1}r^{n-2}|\nabla_\xi r|^2.
\end{equation*}
In the first term of the sum in \eqref{parwj}, we have
\begin{equation*}
    \begin{split}
        &S_0(D^2r)=1,\quad S_1(D^2r)=t\Delta_\xi V+\frac{t^2}{2}\Delta_\xi A+o(t^2),\\
        &S_2(D^2r)=\frac{1}{2}(S_1(D^2r))^2-\sum_{i,j}(D^2r_{ij})^2=t^2S_2(D^2V)+o(t^2),\\
        &S_m(D^2r)=o(t^2),\quad \forall\,m\ge 3.
    \end{split}
\end{equation*}
Then, we just have to consider
\begin{equation*}
    \frac{\binom{n-1}{j-1}\frac{n+j-2}{n-1}r^{n-2}|\nabla_\xi r|^2+\sum_{m=0}^{2}(-1)^m\binom{n-1-m}{j-1-m}r^{n-m}S_m(D^2r)}{(r^2+|\nabla_\xi r|^2)^\frac{j}{2}}=:\frac{f}{g}.
\end{equation*}

The numerator can be expanded as follows
\begin{equation*}
    \begin{split}
         f&=\binom{n-1}{j-1}\frac{n+j-2}{n-1}r^{n-2}|\nabla_\xi r|^2+\binom{n-1}{j-1}r^nS_0(D^2r)-\binom{n-2}{j-2}r^{n-1}S_1(D^2r)\\
         &+\binom{n-3}{j-3}r^{n-2}S_2(D^2r)+o(t^2)\\
         &=\binom{n-1}{j-1}+t\left[n\binom{n-1}{j-1}V-\binom{n-2}{j-2}\Delta_\xi V\right]\\[1.5ex]
         &+\frac{t^2}{2}\left[n(n-1)\binom{n-1}{j-1}V^2+n\binom{n-1}{j-1}A-2(n-1)\binom{n-2}{j-2}V\Delta_\xi V-\binom{n-2}{j-2}\Delta_\xi A\right]\\[1.5ex]
         &+\frac{t^2}{2}\left[2\binom{n-3}{j-3}S_2(D^2V)+ 2\frac{n+j-2}{n-1}\binom{n-1}{j-1}|\nabla_\xi V|^2\right]+o(t^2),
    \end{split} 
\end{equation*}
where we are using the convention that $\binom{n}{s}=0$, if $s<0$.
Now we expand $g$. We start with
\begin{equation*}
    (r^2+|\nabla_\xi r|^2)^\frac{1}{2}=1+tV+\frac{t^2}{2}(A+|\nabla_\xi V|^2)+o(t^2),
\end{equation*}
and then
\begin{equation*}
    \begin{split}
        g&=(r^2+|\nabla_\xi r|^2)^\frac{j}{2}=\left\{1+\left[tV+\frac{t^2}{2}(A+|\nabla_\xi V|^2) \right] \right\}^{j}+o(t^2)\\
        &=\sum_{m=0}^{j}\binom{j}{m}\left[tV+\frac{t^2}{2}(A+|\nabla_\xi V|^2) \right]^m+o(t^2)\\
        &=1+tjV+\frac{t^2}{2}[j(A+|\nabla_\xi V|^2)+j(j-1)V^2]+o(t^2).
    \end{split}
\end{equation*}

Finally,
\begin{equation}
    \label{1}
    \begin{aligned}
\frac{f}{g}
&=
\binom{n-1}{j-1}
+t\left[
(n-j)\binom{n-1}{j-1}V
-\binom{n-2}{j-2}\Delta_\xi V
\right]\\
&+\frac{t^2}{2}\Bigg[
(n-j)(n-j-1)\binom{n-1}{j-1}V^2+(n-j)\binom{n-1}{j-1}A\\
&-2(n-j-1)\binom{n-2}{j-2}
V\Delta_\xi V -\binom{n-2}{j-2}\Delta_\xi A\\
&+2\binom{n-3}{j-3}S_2(D^2V) +\frac{2n+3j-4-nj}{n-1}
\binom{n-1}{j-1}|\nabla_\xi V|^2
\Bigg]+o(t^2).
\end{aligned}
\end{equation}
 
Applying \emph{the Green-Beltrami identity}, we have
\begin{equation}\label{Green-Beltrami}
    \int_{\partial B_1}\Delta_\xi V\,d\mathcal{H}_\xi^{n-1}=\int_{\partial B_1}\Delta_\xi A\,d\mathcal{H}_\xi^{n-1}=0.
\end{equation}
We recall (see \cite{VW} for the proof of this result)
\begin{equation*}
    \int_{\partial B_1}S_k(D^2V)\,d\mathcal{H}^{n-1}=\frac{n-k}{k}\int_{\partial B_1}V^iV_j[T_{k-2}]_i^j(D^2V)\,d\mathcal{H}^{n-1},\quad k\geq 2.
\end{equation*}
Then, for $k=2$, we get
\begin{equation}\label{S_2}
    \int_{\partial B_1}S_2(D^2V)\,d\Hh^{n-1}=\frac{n-2}{2}\int_{\partial B_1}\abs{\nabla_\xi V}^2\,d\Hh^{n-1}_\xi.
\end{equation}

Joining \eqref{1}, \eqref{Green-Beltrami} and \eqref{S_2}, we get the claim.
\end{proof}

We observe that the quermassintegral constraint $W_j(\Omega_t)=C$ gives

$$\frac{d W_j(\Omega_t)}{d t}\bigg|_{t=0}=0, \qquad \frac{d^2 W_j(\Omega_t)}{d t^2}\bigg|_{t=0}=0.$$

Then Theorem \ref{exp-querm} allow us to make explicit the previous conditions
\begin{align}
        \frac{d W_j(\Omega_t)}{d t}\bigg|_{t=0}=0\iff\int_{\p B_1} V(\xi)\,d\Hh_\xi^{n-1}&=0, \tag{I}\label{eq:I}\\
        \frac{d^2 W_j(\Omega_t)}{d t^2}\bigg|_{t=0}=0\iff\int_{\p B_1} A(\xi)\,d\Hh_\xi^{n-1}
        &=-\int_{\p B_1}\left[(n-j-1)V(\xi)^2+\frac{j}{n-1}\abs{\nabla_\xi V}^2\right]\,d\Hh^{n-1}. \tag{II}\label{eq:II}
\end{align}
In the next result we write the Taylor expansion of the curvature integral that appears in \eqref{kaffine}.
\begin{teorema}
Let $\{\Omega(t)\}_t$ be the family of open, bounded convex sets of $\R^n$ defined in \eqref{omegat}, with $\Omega(0)=B_1$. Let \(j=2,\ldots,n\), and assume
that \(\partial\Omega(t)\) is represented by
\[
r(\xi,t)=1+tV(\xi)+\frac{t^2}{2}A(\xi)+o(t^2),
\]
where \(o(t^2)\) is uniform in \(\xi\). Then, 
we have
\begin{equation}\label{IntCurvExp}
\begin{split}
\int_{\partial\Omega(t)}
H_{j-1}^{\frac1{j+1}}\,d\mathcal H^{n-1}
&=
n\omega_n
+
t\,\frac{nj+n-2j}{j+1}
\int_{\partial B_1}
V\,d\mathcal H^{n-1}_\xi
\\
&
+
\frac{t^2}{2}
\int_{\partial B_1}
\Bigg[
\frac{(nj+n-2j)(nj+n-3j-1)}{(j+1)^2}V^2
+
\frac{nj+n-2j}{j+1}A
\\
&
+
\frac{j\bigl(j^2+2jn-6j+2n+1\bigr)}
{(j+1)^2(n-1)}
|\nabla_\xi V|^2
-
\frac{j(j-1)^2}
{(j+1)^2(n-1)^2}
(\Delta_\xi V)^2
\Bigg]
\,d\mathcal H^{n-1}_\xi\\
&+
o(t^2).
\end{split}
\end{equation}
\end{teorema}

\begin{proof}
By Lemma \ref{exp-curv}, for \(H_{j-1}\) we have
\begin{equation*}
\begin{aligned}
H_{j-1}
&=
\binom{n-1}{j-1}^{-1}
\bigl(r^2+|\nabla_\xi r|^2\bigr)^{-\frac{j+1}{2}}
\\
&\quad\times
\sum_{m=0}^{j-1}
(-1)^m
\binom{n-1-m}{j-1-m}
r^{-m}
\left(
r^2S_m(D^2r)
+
\frac{n+j-2-2m}{n-1-m}
r_i r_l [T_m]^l_i(D^2r)
\right).
\end{aligned}
\end{equation*}
As in the proof of Theorem \eqref{exp-querm}, we write
\begin{equation*}
H_{j-1}=\frac{f_1}{g_1},
\end{equation*}
where
\begin{equation*}
\begin{aligned}
f_1
&=
\binom{n-1}{j-1}
\frac{n+j-2}{n-1}
|\nabla_\xi r|^2
+
\sum_{m=0}^{2}
(-1)^m
\binom{n-1-m}{j-1-m}
r^{2-m}S_m(D^2r),
\\
g_1
&=
\binom{n-1}{j-1}
\bigl(r^2+|\nabla_\xi r|^2\bigr)^{\frac{j+1}{2}}.
\end{aligned}
\end{equation*}

Moreover, we recall that
\begin{equation*}
    \begin{split}
        &S_0(D^2r)=1,\quad S_1(D^2r)=t\Delta_\xi V+\frac{t^2}{2}\Delta_\xi A+o(t^2),\\
        &S_2(D^2r)=t^2S_2(D^2V)+o(t^2),\quad S_m(D^2r)=o(t^2),\quad \forall\,m\ge 3.
    \end{split}
\end{equation*}
Thus
\begin{equation*}
\begin{aligned}
f_1
&=
\binom{n-1}{j-1}
+
t
\left[
2\binom{n-1}{j-1}V
-
\binom{n-2}{j-2}\Delta_\xi V
\right]
+
\frac{t^2}{2}
\Bigg[
2\binom{n-1}{j-1}A
+
2\binom{n-1}{j-1}V^2+
\\
&-
2\binom{n-2}{j-2}V\Delta_\xi V
-
\binom{n-2}{j-2}\Delta_\xi A
+
2\binom{n-3}{j-3}S_2(D^2V)+2\binom{n-1}{j-1}
\frac{n+j-2}{n-1}
|\nabla_\xi V|^2
\Bigg]
+
o(t^2).
\end{aligned}
\end{equation*}
Moreover,
\begin{equation*}
\begin{aligned}
g_1
&=
\binom{n-1}{j-1}
\bigl(r^2+|\nabla_\xi r|^2\bigr)^{\frac{j+1}{2}}
\\
&=
\binom{n-1}{j-1}
\Bigg\{
1
+
t(j+1)V
+
\frac{t^2}{2}
\left[
(j+1)A
+
j(j+1)V^2
+
(j+1)|\nabla_\xi V|^2
\right]
\Bigg\}
+
o(t^2).
\end{aligned}
\end{equation*}
Dividing the two expansions, we obtain
\begin{equation*}
\begin{aligned}
H_{j-1}
&=
1
-
t
\left[
(j-1)V
+
\frac{j-1}{n-1}\Delta_\xi V
\right]
\\
&\quad
+
\frac{t^2}{2}
\Bigg[
j(j-1)V^2
-
(j-1)A
-
\frac{j-1}{n-1}\Delta_\xi A
+
\frac{2j(j-1)}{n-1}V\Delta_\xi V
\\
&\quad
+
2
\frac{\binom{n-3}{j-3}}{\binom{n-1}{j-1}}
S_2(D^2V)
-
\frac{(j-1)(n-3)}{n-1}
|\nabla_\xi V|^2
\Bigg]
+
o(t^2).
\end{aligned}
\end{equation*}

Since \(H_{j-1}(B_1)=1\), using Taylor's expansion of the power
\(\frac1{j+1}\), we get
\begin{equation*}
\begin{aligned}
H_{j-1}^{\frac1{j+1}}
&=
1
-
t
\left[
\frac{j-1}{j+1}V
+
\frac{j-1}{(j+1)(n-1)}\Delta_\xi V
\right]
\\
&+
\frac{t^2}{2}
\Bigg[
\frac{2j(j-1)}{(j+1)^2}V^2
-
\frac{j-1}{j+1}A
-
\frac{j-1}{(j+1)(n-1)}\Delta_\xi A
+
\frac{4j(j-1)}{(j+1)^2(n-1)}V\Delta_\xi V
\\
&
+
\frac{2}{j+1}
\frac{\binom{n-3}{j-3}}{\binom{n-1}{j-1}}
S_2(D^2V)
-
\frac{(j-1)(n-3)}{(j+1)(n-1)}
|\nabla_\xi V|^2
-
\frac{j(j-1)^2}{(j+1)^2(n-1)^2}
(\Delta_\xi V)^2
\Bigg]
+
o(t^2).
\end{aligned}
\end{equation*}
We recall that
\begin{equation*}
r^{n-2}
=
1
+
t(n-2)V
+
\frac{t^2}{2}
\left[
(n-2)A
+
(n-2)(n-3)V^2
\right]
+
o(t^2),
\end{equation*}
and
\begin{equation*}
\bigl(r^2+|\nabla_\xi r|^2\bigr)^{1/2}
=
1
+
tV
+
\frac{t^2}{2}
\left[
A
+
|\nabla_\xi V|^2
\right]
+
o(t^2).
\end{equation*}
Hence
\begin{equation*}
r^{n-2}
\bigl(r^2+|\nabla_\xi r|^2\bigr)^{1/2}
=
1
+
t(n-1)V
+
\frac{t^2}{2}
\left[
(n-1)A
+
(n-1)(n-2)V^2
+
|\nabla_\xi V|^2
\right]
+
o(t^2).
\end{equation*}

Multiplying the expansion of \(H_{j-1}^{1/(j+1)}\) by the expansion of
the surface element and integrating, we obtain
\begin{equation*}
\begin{aligned}
&\int_{\partial B_1}H_{j-1}^{\frac1{j+1}}
\,d\mathcal H^{n-1}_{\partial\Omega(t)}
=
\int_{\partial B_1}\Bigg\{
1
+
t
\left[
\frac{nj+n-2j}{j+1}V
-
\frac{j-1}{(j+1)(n-1)}\Delta_\xi V
\right]
\\
&\quad
+
\frac{t^2}{2}
\Bigg[
\frac{(nj+n-2j)(nj+n-3j-1)}{(j+1)^2}V^2
+
\frac{nj+n-2j}{j+1}A
\\
&\qquad
-
\frac{j-1}{(j+1)(n-1)}\Delta_\xi A
-
\frac{2(j-1)(nj+n-3j-1)}
{(j+1)^2(n-1)}
V\Delta_\xi V
\\
&\qquad
+
\frac{2}{j+1}
\frac{\binom{n-3}{j-3}}{\binom{n-1}{j-1}}
S_2(D^2V)
+
\frac{2(n+j-2)}{(j+1)(n-1)}
|\nabla_\xi V|^2
-
\frac{j(j-1)^2}
{(j+1)^2(n-1)^2}
(\Delta_\xi V)^2
\Bigg]
\Bigg\}
d\mathcal H^{n-1}_\xi
+
o(t^2).
\end{aligned}
\end{equation*}
Using \eqref{Green-Beltrami}-\eqref{S_2}, and recalling that
\begin{equation*}
\int_{\partial B_1}
V\Delta_\xi V\,d\mathcal H^{n-1}_\xi
=
-
\int_{\partial B_1}
|\nabla_\xi V|^2\,d\mathcal H^{n-1}_\xi.
\end{equation*}
we finally get
\begin{equation*}
\begin{aligned}
\int_{\partial\Omega(t)}
H_{j-1}^{\frac1{j+1}}\,d\mathcal H^{n-1}
&=
n\omega_n
+
t\,\frac{nj+n-2j}{j+1}
\int_{\partial B_1}
V\,d\mathcal H^{n-1}_\xi
\\
&
+
\frac{t^2}{2}
\int_{\partial B_1}
\Bigg[
\frac{(nj+n-2j)(nj+n-3j-1)}{(j+1)^2}V^2
+
\frac{nj+n-2j}{j+1}A
\\
&
+
\frac{j\bigl(j^2+2jn-6j+2n+1\bigr)}
{(j+1)^2(n-1)}
|\nabla_\xi V|^2
-
\frac{j(j-1)^2}
{(j+1)^2(n-1)^2}
(\Delta_\xi V)^2
\Bigg]
\,d\mathcal H^{n-1}_\xi
+
o(t^2).
\end{aligned}
\end{equation*}
\end{proof}
Now we are in position to prove Theorem \ref{teo_curvature_sella}.
\begin{proof}[proof of Theorem \ref{teo_curvature_sella}]
We want to write each term in \eqref{IntCurvExp} in terms of spherical harmonic functions. We start by writing the normal velocity on $\p B_1$ as
\begin{equation*}
        V(\xi)=\sum_{\ell\ge0}a_\ell Y_\ell(\xi),
        \qquad \|Y_\ell\|_{L^2(\p B_1)}=1,
\end{equation*}
where
\begin{equation}\label{propHarm}
        \Delta_{\xi}Y_\ell(\xi)=-\ell(\ell+n-2)Y_\ell(\xi),
\end{equation}
and
\[
        a_\ell=\int_{\partial B_1}V(\xi)Y_\ell(\xi)\,d\mathcal{H}^{n-1}_\xi.
\]
The quermassintegral constraint \eqref{eq:I} gives
\begin{equation*}
        a_0=0.
\end{equation*}
Moreover,
\begin{equation}\label{eq:phisquare}
        \int_{\p B_1}V(\xi)^2=\sum_{\ell\ge0}a_\ell^2.
\end{equation}
From \eqref{propHarm} we can infer the following
\begin{equation*}
    \int_{\partial B_1}|\nabla_\xi V|^2\,d\mathcal{H}^{n-1}=-\int_{\partial B_1}V\Delta_\xi V\,d\mathcal{H}^{n-1}=\sum_{\ell\geq0}a_\ell^2\ell(\ell+n-2).
\end{equation*}
Combining the last two identity, from \eqref{eq:II}, we find
\begin{equation}\label{eq:Aintegral}
\begin{aligned}
        \int_{\p B_1}A(\xi)\, d\mathcal{H}^{n-1}
        =&-\int_{\p B_1}\left[(n-j-1)V(\xi)^2+\frac{j}{n-1}\abs{\nabla_\xi V}^2\right]\,d\Hh^{n-1}\\
        =& -\sum_{\ell\ge 0} a_\ell^2 \left((n-j-1)+\frac{j}{n-1}\ell(\ell+n-2)\right)
        \end{aligned}
\end{equation}
    Using the spherical harmonics expansion given above in \eqref{IntCurvExp} we get
    \begin{equation*}
        \begin{split}
            \int_{\partial\Om(t)}H_{j-1}^\frac{1}{j+1}\,d\Hh^{n-1}\bigg|_{t=0}=&\sum_{\ell\geq0}a_\ell^2\bigg[\frac{(nj+n-2j)(nj+n-3j-1)}{(j+1)^2}\\
            &-\frac{nj+n-2j}{j+1}\left((n-i-1)+\frac{i}{n-1}\ell(\ell+n-2)\right)\\
            &+\frac{j\bigl(j^2+2jn-6j+2n+1\bigr)}
            {(j+1)^2(n-1)}\ell(\ell+n-2)\\
            &-\frac{j(j-1)^2}
{(j+1)^2(n-1)^2}\ell^2(\ell+n-2)^2\bigg],
        \end{split}
    \end{equation*}
    where the $i$-th quermassintegral is prescribed.
    Then, writing everything in a more compact form, we have
    \begin{equation*}
        \begin{split}
            \int_{\partial\Om(t)}H_{j-1}^\frac{1}{j+1}\,d\Hh^{n-1}\bigg|_{t=0}=&\sum_{\ell\geq0}a_\ell^2f_{n,j,i}(\ell).
        \end{split}
    \end{equation*}
    In particular we are interested in the case $i=0$, namely when the measure is prescribed. In this case we have
    \begin{equation*}
        \begin{split}  
            f_{n,j,0}(\ell)=&-\frac{j(j-1)^2}{(j+1)^2(n-1)^2}\ell^4
-\frac{2j(j-1)^2(n-2)}{(j+1)^2(n-1)^2}\ell^3\\
&-\frac{
j^3 n^2-5j^3 n+5j^3
-4j^2 n^2+16j^2 n-14j^2
-jn^2-3jn+5j
}
{(j+1)^2(n-1)^2}\ell^2\\
&+\frac{(n-2)\bigl(
j^3+2j^2n-6j^2+2jn+j
\bigr)}
{(j+1)^2(n-1)}\ell
-\frac{2j(jn-2j+n)}{(j+1)^2}.
        \end{split}
    \end{equation*}
    It is easy to see that $f_{n,j,0}(2)>0$, while exists $\overline{\ell}\geq3$ such that $f_{n,j,0}(\overline{\ell})<0$ for all $j\not=1,n$. In the case $j=n$ we have $f_{n,n,0}(2)=0$, accordingly to the affine-invariance of the functional, and $f_{n,n,0}(\ell)>0$ for all $\ell>2$, accordingly to \eqref{AffIsopIneq}.
\end{proof}

\section{Shape derivative of the $k$-Torsional rigidity}\label{sec:prove}
We now want to compute the first and the second derivatives of the Torsion problem

\begin{equation}\label{eq:P}
\begin{cases}
S_k(D^2u)=\displaystyle \frac1k\bigl(S_k^{ij}(D^2u)u_j\bigr)_i=\Binom{n}{k} &\text{in }\Om(t),\\[0.4em]
 u=0 &\text{on }\p\Om(t).
\end{cases}
\tag{P}
\end{equation}
where we denote $u=u(x,t)$.
To simplify the notation, we omit to specify the variables, so when $\displaystyle{\frac{\partial \Phi}{\partial t}}$ appears, one has to intend  $\displaystyle{\frac{\partial \Phi(\Phi^{-1}(x))}{\partial t}}$, with $x\in\Omega(t)$.

\vspace{2mm}
Differentiating \eqref{eq:P} once with respect to $t$ gives
\begin{equation}\label{eq:Pt}
\begin{cases}
\bigl(S_k^{ij}(D^2u)u_{tj}\bigr)_i=0 &\text{in }\Om(t),\\[1ex]
\displaystyle{ u_t+\grad u\cdot \frac{\partial \Phi}{\partial t}}=0 &\text{on }\p\Om(t).
\end{cases}
\tag{$P_t$}
\end{equation}
Differentiating twice gives
\begin{equation}\label{eq:Ptt}
\begin{cases}
\bigl(S_k^{ij}(D^2u)\bigr)_t u_{tij}+\bigl( S_k^{ij}(D^2u)u_{ttj}\bigr)_i=0
&\text{in }\Om(t),\\[2ex]
\displaystyle{ u_{tt}+\left(\dfrac{\partial\Phi}{\partial t}\right)^T D^2u\,\frac{\partial\Phi}{\partial t}+\grad u\cdot \frac{\partial^2\Phi}{\partial t^2}+2\grad u_t\cdot \frac{\partial\Phi}{\partial t}=0}
&\text{on }\p\Om(t).
\end{cases}
\tag{$P_{tt}$}
\end{equation}

On the boundary, \eqref{eq:Pt} gives
\begin{equation*}
        V\big|_{\p\Om(t)}=-\frac{u_t}{|\grad u|}\,\nuout .
\end{equation*}
In particular, at $t=0$, 
\begin{equation}\label{eq:ut-boundary}
        u_t\big|_{\p B_1}=-V .
\end{equation}

We also use the Hadamard formula

\begin{equation}\label{eq:Hadamard}
    \frac{d}{dt}\int_{\Phi(t,\Om)}f\,dx=\int_{\Phi(t,\Om)}\left[f_t+\mathrm{div}\left(f\frac{\partial\Phi}{\partial t}\right)\right]\,dx.
\end{equation}

We are now in position to prove that the ball is a critical point for the $k$-Torsional rigidity when prescribing any quermassintegral $W_j$.
\begin{proof}[proof of Theorem \ref{derivata1}]
Consider the functional
\begin{equation*}
        T_k(t)
        :=\frac{\left(\displaystyle\int_{\Om(t)}-u\,dx\right)^{k+1}}
        {\displaystyle\int_{\Om(t)}-u\,S_k(D^2u)\,dx}.
\end{equation*}
Using the equation in \eqref{eq:P}, the last becomes
\begin{equation*}
        T_k(t)=\Binom{n}{k}^{-1}
        \left(\int_{\Om(t)}-u\,dx\right)^k .
\end{equation*}
From \eqref{eq:Hadamard} and $u=0$ on $\p\Om(t)$, we can compute the first shape derivative of the $k$-torsional rigidity as follows
\begin{equation}
   \label{der_tor}
   \begin{aligned}
       T_k'(t)
&=k\Binom{n}{k}^{-1}\left(\int_{\Om(t)}-u\,dx\right)^{k-1}
        \int_{\Om(t)}-u_t -\text{div}\left(u\frac{\partial\Phi}{\partial t}\right)\, dx\\
        &=k\Binom{n}{k}^{-1}\left(\int_{\Om(t)}-u\,dx\right)^{k-1}
        \int_{\Om(t)}-u_t\,dx.
   \end{aligned}
\end{equation}

Since
\[
        1=\frac{1}{k\Binom{n}{k}}\bigl(S_k^{ij}(D^2u)u_j\bigr)_i,
\]
we get
\begin{align*}
T_k'(t)
&=\Binom{n}{k}^{-2}\left(\int_{\Om(t)}-u\,dx\right)^{k-1}
 \int_{\Om(t)}-u_t\bigl(S_k^{ij}(D^2u)u_j\bigr)_i\,dx \\
&=\Binom{n}{k}^{-2}\left(\int_{\Om(t)}-u\,dx\right)^{k-1}
 \int_{\p\Om(t)}\frac{-u_tS_k^{ij}(D^2u)u_iu_j}{|\grad u|}\,d\Hh^{n-1}
 \\
&\quad+\Binom{n}{k}^{-2}\left(\int_{\Om(t)}-u\,dx\right)^{k-1}
 \int_{\Om(t)} u_jS_k^{ij}(D^2u)u_{ti}\,d\Hh^{n-1}.
\end{align*}
The last volume term vanishes by \eqref{eq:Pt} after integration by parts, because
$u=0$ on $\p\Om(t)$. Moreover, on $\p\Om(t)$, \eqref{hksk} is in force, 
\begin{equation*}
        \frac{S_k^{ij}(D^2u)u_iu_j}{|\grad u|}
        =\Binom{n-1}{k-1}H_{k-1}|\grad u|^k ,
\end{equation*}
hence,
\begin{align}\label{eq:Tprime}
T_k'(t)
&=\Binom{n-1}{k-1}\Binom{n}{k}^{-2}
 \left(\int_{\Om(t)}-u\,dx\right)^{k-1}
 \int_{\p\Om(t)}-u_tH_{k-1}|\grad u|^k\,d\Hh^{n-1} \\
&=\frac{k}{n}\Binom{n}{k}^{-1}
 \left(\int_{\Om(t)}-u\,dx\right)^{k-1}
 \int_{\p\Om(t)}H_{k-1}|\grad u|^{k+1}V\,d\Hh^{n-1} . \nonumber
\end{align}
On  $\p B_1$, $H_{k-1}$ and $|\grad u|$ are constant, hence
\[
        T_k'(0)=C\int_{\p B_1}V\,d\Hh^{n-1}=0
\]
by \eqref{eq:I}. Thus the ball is a critical shape.
\end{proof}

In order to prove Theorem \ref{teo_sella}, we compute the second shape derivative of $T_k$.
\begin{lemma}[Second derivative]
    Let $\Phi:[0,\delta)\times \mathbb{R}^n\rightarrow\mathbb{R}^n$ be a family of transformations satisfying, for $\delta>0$ small enough, $(i)-(iv)$. Then, it holds
    \begin{equation*}
        \begin{split}
            T_k''(t)=&\frac{k(k-1)}{n^2}\Binom{n}{k}^{-1}
            \left(\int_{\Om(t)}-u\,dx\right)^{k-2}
            \left(\int_{\p\Om(t)}H_{k-1}|\grad u|^{k+1}V\,d\Hh^{n-1}\right)^2\\
            &+k\Binom{n}{k}^{-1}\left(\int_{\Om(t)}-u\,dx\right)^{k-1}
            \Bigg\{
            \int_{\p\Om(t)}|\grad u|\,V^2 \,d\Hh^{n-1}+\frac1k\Binom{n}{k}^{-1}
            \int_{\Om(t)}\bigl(S_k^{ij}(D^2u)\bigr)_t u_{tij}\,u\,dx  \nonumber\\
            &+\frac1n\int_{\p\Om(t)}H_{k-1}|\grad u|^k
            \left[\left(\frac{\partial\Phi}{\partial t}\right)^TD^2u\,\frac{\partial\Phi}{\partial t}+\grad u\cdot \frac{\partial^2\Phi}{\partial t^2}+2\grad u_t\cdot \frac{\partial\Phi}{\partial t}\right]\,d\Hh^{n-1}
            \Bigg\}.
        \end{split}
    \end{equation*}
    Moreover, when $t=0$, we get
    \begin{equation}\label{T''(0)}
        \begin{split}
            T_k''(0)=&\frac{k(k-1)}{n^2}\Binom{n}{k}^{-1}
            \left(\int_{B_1}-u\,dx\right)^{k-2}
            \left(\int_{\p B_1}V\,d\Hh^{n-1}\right)^2\\
            &+k\Binom{n}{k}^{-1}\left(\int_{B_1}-u\,dx\right)^{k-1}
            \Bigg\{
            \int_{\p B_1}V^2 \,d\Hh^{n-1}+\frac1k\Binom{n}{k}^{-1}
            \int_{B_1}\bigl(S_k^{ij}(D^2u)\bigr)_t u_{tij}\,u\bigg|_{t=0}\,dx \\
            &+\frac1n\int_{\p B_1}
            \left[V^2+A+2V\frac{\partial u}{\partial \nu}\bigg|_{t=0}\right]\,d\Hh^{n-1}
            \Bigg\}.
        \end{split}
    \end{equation}
\end{lemma}
\begin{proof}
We want to compute $T_k''(t)$. 
From \eqref{eq:Tprime}, we just need to compute the following 
\begin{align*}
I_k''(t)
&=k\Binom{n}{k}^{-1}\left(\int_{\Om(t)}-u\,dx\right)^{k-1}
        \frac{d}{dt}\int_{\Om(t)}-u_t\,dx,
\end{align*}
and this can be done by differentiating the expression \eqref{der_tor}, obtaining
 \begin{equation}\label{eq:Tsecond-start}
I''_k(t)=k\Binom{n}{k}^{-1}\left(\int_{\Om(t)}-                 u\,dx\right)^{k-1}\int_{\Om(t)}\left[-u_{tt}-\Div\left(u_t\frac{\partial\Phi}{\partial t}\right)\right]\,dx. 
 \end{equation}

The divergence term in \eqref{eq:Tsecond-start} is
\begin{equation*}
        \int_{\Om(t)}-\Div\left(u_t\frac{\partial\Phi}{\partial t}\right)\,dx
        =\int_{\p\Om(t)}-u_tV\,d\Hh^{n-1}
        =\int_{\p\Om(t)}|\grad u|\,V^2\,d\Hh^{n-1} .
\end{equation*}
For the term involving $u_{tt}$, applying the divergence Theorem and \eqref{hksk},
\begin{align*}
\int_{\Om(t)}-u_{tt}\,dx
&=\frac{1}{k}\Binom{n}{k}^{-1}
\int_{\Om(t)}-u_{tt}\bigl(S_k^{ij}(D^2u)u_j\bigr)_i\,dx \\
&=\frac{1}{k}\Binom{n}{k}^{-1}\Binom{n-1}{k-1}
\int_{\p\Om(t)}-u_{tt}H_{k-1}|\grad u|^k\,d\Hh^{n-1}\quad+\frac{1}{k}\Binom{n}{k}^{-1}
\int_{\Om(t)}S_k^{ij}(D^2u) u_{tti}\,u_j\,dx.
\end{align*}
If we integrate by part and we recall the equation in \eqref{eq:Ptt}, we have 
$$\begin{aligned}
\int_{\Om(t)}S_k^{ij}(D^2u) u_{tti}\,u_j \,dx&=\int_{\partial\Om(t)}S_k^{ij}(D^2u) u_{tti}\nu_j u\,d\Hh^{n-1} -\int_{\Om(t)}\left(S_k^{ij}(D^2u) u_{tti}\right)_j u\,dx\\
&=\int_{\Om(t)} \left(S_k^{ij}(D^2u)\right)_t u_{tij} u\,dx.
\end{aligned}$$
Using the boundary condition in \eqref{eq:Ptt}, namely
\[
        -u_{tt}=\left(\frac{\partial\Phi}{\partial t}\right)^TD^2u\,\frac{\partial\Phi}{\partial t}+\grad u\cdot \frac{\partial^2\Phi}{\partial t^2}+2\grad u_t\cdot \frac{\partial\Phi}{\partial t}
        \qquad\text{on }\p\Om(t),
\]
and the identity $\Binom{n-1}{k-1}/(k\Binom nk)=1/n$, we arrive at
\begin{align*}
I_k''(t)
&=k\Binom{n}{k}^{-1}\left(\int_{\Om(t)}-u\,dx\right)^{k-1}
\Bigg\{
\int_{\p\Om(t)}|\grad u|\,V^2\,d\Hh^{n-1} \\
&\quad+\frac1k\Binom{n}{k}^{-1}
\int_{\Om(t)}\bigl(S_k^{ij}(D^2u)\bigr)_t u_{tij}\,u \,dx \nonumber\\
&\quad+\frac1n\int_{\p\Om(t)}H_{k-1}|\grad u|^k
\left[\left(\frac{\partial\Phi}{\partial t}\right)^TD^2u\,\frac{\partial\Phi}{\partial t}+\grad u\cdot \frac{\partial^2\Phi}{\partial t^2}+2\grad u_t\cdot \frac{\partial\Phi}{\partial t}\right]\,d\Hh^{n-1}
\Bigg\}. \nonumber
\end{align*}
This complete the first part of the statement.\\ Now we want to evaluate $T_k''$ at $t=0$.
We recall that, at $t=0$, the solution to \eqref{eq:P} is 
\begin{equation*}
        u(x)=\frac{|x|^2-1}{2},\qquad D^2u=\mathbb{I},
        \qquad |\grad u|=1\quad\text{on }\p B_1.
\end{equation*}
Moreover,
\begin{equation*}
    H_{k-1}\big|_{\p B_1}=1.
\end{equation*}
Since $D^2u=\mathbb{I}$, on $\p B_1$, 
\begin{equation*}
        \left(\frac{\partial\Phi}{\partial t}\right)^TD^2u\,\frac{\partial\Phi}{\partial t}\bigg|_{t=0}=V^2 .
\end{equation*}

Also, using \eqref{eq:ut-boundary},
\begin{equation*}
        2\grad u_t\cdot \frac{\partial\Phi}{\partial t}
        =2V\frac{\partial u_t}{\partial\nuout}
        =-2u_t\frac{\partial u_t}{\partial\nuout}
        \qquad\text{on }\p B_1.
\end{equation*}
At $t=0$, the differentiated equation \eqref{eq:Pt} reduces to
\begin{equation*}
        \Delta u_t=0\qquad\text{in }B_1.
\end{equation*}
This concludes the proof.
\end{proof}

\begin{proof}[proof of Theorem \ref{teo_sella}]
We introduce for $x\equiv (r,\xi)\in\Omega$
\begin{equation*}
    v(r,\xi)=u_t(x),
\end{equation*}
Now we need an expression for $v(r,\xi)$. We observe that the equation for $u_t$ is $S_k^{ij}(D^2u)u_{tij}=0$ and, at $t=0$, since $S_k^{ij}(D^2u)=\delta_{ij}$, becomes $\Delta u_t=0$. Then in spherical coordinates we have
\begin{equation*}
    \begin{cases}
        -\frac{1}{r^{n-1}}\frac{\partial}{\partial r}\left(r^{n-1}\frac{\partial v}{\partial r}\right)-\frac{1}{r^2}\Delta_\xi v=0 & (r,\xi)\in(0,1)\times \partial B_1\\
        v(1,\xi)=-V(\xi) & \xi\in \partial B_1.
    \end{cases}
\end{equation*}
Then, with the separation of variables
\begin{equation}\label{varsep}
    v(r,\xi)=\sum_{\ell=0}^{+\infty}R_\ell(r)Y_\ell(\xi),
\end{equation}
the equation becomes
\begin{equation*}
    -\frac{n-1}{r}\frac{\partial v}{\partial r}-\frac{\partial^2 v}{\partial r^2}-\frac{1}{r^2}\Delta_\xi v=0,
\end{equation*}
that is from \eqref{varsep}
\begin{equation*}
    \sum_{\ell=0}^{+\infty} \left[-\frac{n-1}{r}R'_\ell(r)Y_\ell(\xi)-R''_\ell(r)Y_\ell(\xi)-\frac{1}{r^2}R_\ell(r)\Delta_\xi Y_\ell(\xi) \right]=0.
\end{equation*}
Using \eqref{propHarm} we get
\begin{equation*}
    \sum_{\ell=0}^{+\infty} \left[-\frac{n-1}{r}R'_\ell(r)-R''_\ell(r)+\frac{\ell(\ell+n-2)}{r^2}R_\ell(r) \right]Y_\ell(\xi)=0,
\end{equation*}
but, since $\{Y_\ell\}_\ell$ is a basis of $L^2$, then we have that, for all $\ell\geq 0$,
\begin{equation*}
    \begin{cases}
        &R''_\ell(r)+\frac{n-1}{r}R'_\ell(r)-\frac{\ell(\ell+n-2)}{r^2}R_\ell(r)=0,\\
        &R_\ell(1)=-a_\ell,\\
        &R'_\ell(0)=0.
    \end{cases}
\end{equation*}
Explicitly solving the ODE, we arrive to
\begin{equation*}
    R_\ell(r)=c_1r^\ell+c_2r^{2-(n+\ell)},
\end{equation*}
with $c_1$, $c_2$ to be determined. Since we are interested in solutions that are $L^2$, we choose $c_2=0$ and then, from $R_\ell(1)=-a_\ell$ we find
\begin{equation*}
    R_\ell(r)=-a_\ell r^\ell.
\end{equation*}
Now we are interested in the radial derivative of $v(r,\xi)$, evaluated in $r=1$, i.e.,
\begin{equation*}
    \frac{\partial v}{\partial r}(r,\xi)\bigg|_{r=1}=\sum_{\ell=0}^{+\infty} R'_\ell(1)Y_\ell(\xi)=-\sum_{\ell=0}^{+\infty}\ell a_\ell Y_\ell(\xi),
\end{equation*}
Therefore
\begin{equation*}
        \int_{\p B_1}2V\frac{\partial u}{\partial \nu}\bigg|_{t=0}\,d\Hh^{n-1}
        =-2\sum_{\ell\ge0}\ell a_\ell^2.
\end{equation*}
Combining the previous identity with \eqref{eq:phisquare}-\eqref{eq:Aintegral}, we find that the purely boundary contribution in \eqref{T''(0)} is
\begin{align}\label{eq:boundary-contribution}
&\int_{\p B_1}V^2 \,d\Hh_\xi^{n-1}
+\frac1n\int_{\p B_1}\bigl[V^2+A+2\grad u_t\cdot \frac{\partial\Phi}{\partial t}\bigr] \\
&\qquad
=\sum_{\ell\ge0}a_\ell^2\left(\frac{j+2}{n}-\frac{j(n-2)+2(n-1)}{n(n-1)}\ell-\frac{j}{n(n-1)}\ell^2\right). \nonumber
\end{align}
The last term that we need to write in terms of spherical harmonics in \eqref{T''(0)} is the Newton-tensor term. We start by recalling the identity \eqref{cofproof}
\begin{equation*}
    S_k^{ij}(A)=[T_{k-1}]_i^j(A),
\end{equation*}
that in our case is
\begin{equation*}
    S_k^{ij}(D^2u)=[T_{k-1}]_i^j(D^2u)=\frac{1}{(k-1)!}\delta_{jj_1\dots j_{k-1}}^{ii_1\dots i_{k-1}}u_{i_1j_1}\dots u_{i_{k-1}j_{k-1}}.
\end{equation*}
Then
\begin{equation*}
    \begin{split}
        (S_k^{ij}(D^2u))_tu_{ijt}&=\left([T_{k-1}]_i^j(D^2u)\right)_tu_{ijt}=\frac{1}{(k-1)!}\delta_{jj_1\dots j_{k-1}}^{ii_1\dots i_{k-1}}(u_{i_1j_1}\dots u_{i_{k-1}j_{k-1}})_tu_{ijt}\\
        &=\frac{k-1}{(k-1)!}\delta_{jj_1\dots j_{k-1}}^{ii_1\dots i_{k-1}}(u_{i_1j_1})_tu_{i_{2}j_{2}}\dots u_{i_{k-1}j_{k-1}}u_{ijt}.
    \end{split}
\end{equation*}
We remark that on the ball $\Omega(0)=B_1$ the solution $u$ to \eqref{torsDirich} satisfies $u_{ij}=\delta_{ij}$. Then, iterating \eqref{contr2}, we have
\begin{equation*}
    \delta^{ii_1\dots i_{k-1}}_{jj_1\dots j_{k-1}}\delta_{i_2}^{j_2}\dots\delta^{j_{k-1}}_{i_{k-1}}=(n-k+1)(n-k+2)\dots(n-2)\delta_{jj_1}^{ii_1}
\end{equation*}

Then, on $B_1$, we have
\begin{equation*}
    (S_k^{ij}(D^2u))_tu_{ijt}=\frac{(n-k+1)(n-k+2)\dots(n-2)}{(k-2)!}\delta_{ii_1}^{jj_1}(u_{i_1j_1})_t(u_{ij})_t=\binom{n-2}{k-2}\left[(\Delta u_t)^2-\sum_{ij}u_{ijt}^2\right].
\end{equation*}
Since $u_t$ is harmonic in $B_1$, applying the divergence Theorem, we find
\begin{equation*}
    \binom{n-2}{k-2}^{-1}\int_{B_1}(S_k^{ij}(D^2u))_tu_{tij}u\,dx=-\sum_{i,j}\int_{B_1}u^2_{tij}u\,dx=\int_{B_1}u_{tij}u_ju_{ti}\,dx=\frac{1}{2}\int_{B_1} (u_{ti}^2)_j u_j\,dx.
\end{equation*}
Now applying again the divergence Theorem, since on $\partial B_1$ it holds $|\nabla u|=1$ and $\Delta u=n$, 
\begin{equation}\label{Snesp}
    \binom{n-2}{k-2}^{-1}\int_{B_1}(S_k^{ij}(D^2u))_tu_{tij}u\,dx=\frac{1}{2}\sum_i\int_{\partial B_1}u_{ti}^2\,d\mathcal{H}^{n-1}-\frac{n}{2}\sum_i\int_{B_1}u^2_{ti}\,dx.
\end{equation}
The second term in the last equality, integrating by parts, is
\begin{equation}\label{secter}
\begin{split}
    \sum_i\int_{B_1}u_{ti}^2\,dx&=\int_{\partial B_1}u_t\frac{\partial u_t}{\partial\nu}\,d\mathcal{H}^{n-1}=-\int_{\partial B_1}\frac{\partial u_t}{\partial\nu}V\,d\mathcal{H}^{n-1}\\
    &=-\int_{\partial B_1}\frac{\partial v}{\partial r}(r,\xi)\bigg|_{r=1}V(\xi)\,d\mathcal{H}^{n-1}_\xi=\sum_{\ell\geq 0}a_\ell^2\ell,
\end{split}    
\end{equation}
where, again, we have used $\Delta u_t=0$.
For the first term in \eqref{Snesp} we recall the following integration by parts formula (see \cite[Theorem 5.4.13]{HP})
\begin{equation}\label{intXpar}
    \int_{\partial\Omega}\nabla f\cdot \nabla g\,d\mathcal{H}^{n-1}=-\int_{\partial\Omega}f\Delta g\,d\mathcal{H}^{n-1}+\int_{\partial\Omega}\left(\frac{\partial f}{\partial\nu} \frac{\partial g}{\partial\nu}+f\frac{\partial^2 g}{\partial\nu^2}+\tilde H_1f\frac{\partial g}{\partial\nu}\right)\,d\mathcal{H}^{n-1}.
\end{equation}
Then, from \eqref{intXpar}, recalling that $u_t$ is harmonic, we have
\begin{equation}\label{priter}
    \begin{split}
         \int_{\partial B_1}\nabla & u_t\cdot \nabla u_t\,d\mathcal{H}^{n-1}=\int_{\partial B_1}\left(\frac{\partial u_t}{\partial\nu}\right)^2\,d\mathcal{H}^{n-1}+\int_{\partial B_1}u_t\frac{\partial^2 u_t}{\partial\nu^2}\,d\mathcal{H}^{n-1}+\int_{\partial B_1}\tilde H_1u_t\frac{\partial u_t}{\partial\nu}\,d\mathcal{H}^{n-1}\\
         &=\int_{\partial B_1}\left[\left(\frac{\partial v}{\partial r}(r,\xi)\bigg|_{r=1}\right)^2-V(\xi) \frac{\partial^2 v}{\partial r^2}(r,\xi)\bigg|_{r=1}-(n-1)V(\xi) \frac{\partial v}{\partial r}(r,\xi)\bigg|_{r=1}\right]\,d\mathcal{H}^{n-1}_\xi\\
         &=\sum_{\ell\geq0}a_\ell^2[\ell^2+\ell(\ell-1)+\ell(n-1)].
    \end{split}
\end{equation}
then, from \eqref{Snesp}-\eqref{secter}-\eqref{priter}, we have
\begin{equation}\label{NewTensTerm}
    \int_{B_1}(S_k^{ij}(D^2u))_tu_{tij}u=\binom{n-2}{k-2}\sum_{\ell\geq 0}a_\ell^2(\ell^2-\ell).
\end{equation}
Substituting \eqref{eq:boundary-contribution} and \eqref{NewTensTerm} into
\eqref{T''(0)}, we obtain
\begin{equation*}
        T_k''(0)\bigg|=C_{n,k}\sum_{\ell\ge0}a_\ell^2 f_{n,k,j}(\ell),
\end{equation*}
where the $j$-th quermassintegral is prescribed, $C_{n,k}>0$ and
\begin{equation*}
    f_{n,k,j}(\ell)=\frac{k-1-j}{n\left(n-1\right)}\ell^{2}-\frac{k+2n+j\left(n-2\right)-3}{n\left(n-1\right)}\ell+\frac{2+j}{n}.
\end{equation*}
The results can be summarised as follows
\begin{itemize}
    \item [($k=1$)] \textbf{Laplacian}
    \begin{itemize}
        \item $f_{n,1,j}(1)=0$, for all $j=0,\dots,n-1$, accordingly with translation invariance;
        \item $f_{n,1,j}(\ell)<0$, for all $j=0,\dots,n-1$, and for all $\ell>1$, accordingly with \eqref{SaintFenchel};
    \end{itemize}
    \item [($1<k<n$)] $\mathbf{k}$\textbf{-Hessian}
    \begin{itemize}
        \item $f_{n,k,j}(1)=0$, for all $j=0,\dots,n-1$, accordingly with the translation invariance;
        \item $f_{n,k,j}(2)>0$ and there exists $\overline{\ell}>2$ such that $f_{n,k,j}(\ell)<0$ for all $j=0,\dots,k-2$, and for all $\ell>\overline{\ell}$, then the ball is a saddle point in this regime;
        \item $f_{n,k,j}(\ell)<0$ for all $j=k-1,\dots,n-1$, and for all $\ell>1$, accordingly with \cite{tso} and \eqref{Aleksandrov_Fenchel_inequalities};
    \end{itemize}
    \item [($k=n$)] \textbf{Monge-Ampère}
    \begin{itemize}
        \item $f_{n,n,0}(1)=f_{n,n,0}(2)=0$, accordingly with the translation and affinity invariance;
        \item $f_{n,n,0}(\ell)>0$ for all $\ell>2$, accordingly with \cite{BNT_newisoperimetric};
        \item $f_{n,n,j}(1)=0$ for all $j=0,\dots,n-1$, accordingly with the translation invariance;
        \item $f_{n,n,j}(2)>0$ and there exists $\overline{\ell}>2$ such that $f_{n,n,j}(\ell)<0$ for all $j=1,\dots,n-2$, and for all $\ell>\overline{\ell}$, then the ball is a saddle point in this regime;
        \item $f_{n,n,n-1}(\ell)<0$ for all $\ell>1$, accordingly with \cite{tso}.
    \end{itemize}
\end{itemize}

This concludes the proof.
\end{proof}

\paragraph{Funding}
The authors were partially supported by Gruppo Nazionale per l’Analisi Matematica, la Probabilità e le loro Applicazioni
(GNAMPA) of Istituto Nazionale di Alta Matematica (INdAM).

Francesco Salerno is partially supported by INdAM GNAMPA 2026 Project ``Processi di diffusione non-lineari: regolarità e classificazione delle soluzioni'', CUP E53C25002010001

Alba Lia Masiello and Gloria Paoli were supported by COST Action 24122 mSPACE, supported by COST (European Cooperation
in Science and Technology), www.cost.e
\paragraph{Competing Interests}
We declare that we have no financial and personal relationships with other people or organizations.
\paragraph{AI statement}
Generative AI tools were used solely to check the manuscript for typographical errors and minor language issues. No generative AI tool was used to generate research ideas, interpret results, or write scientific content.


 \bibliographystyle{plain}
\bibliography{biblio}
\addcontentsline{toc}{chapter}{Bibliography}
\Addresses
\end{document}